\documentclass[11pt]{article}

\usepackage[margin=1in]{geometry}
\usepackage{amsmath,amssymb,amsthm,mathtools}
\usepackage{bm}
\usepackage{xcolor}
\usepackage{enumitem}
\usepackage{graphicx}
\usepackage{hyperref}

\hypersetup{
  colorlinks=true,
  linkcolor=blue,
  citecolor=blue,
  urlcolor=blue
}

\newtheorem{theorem}{Theorem}[section]
\newtheorem{proposition}[theorem]{Proposition}

\newtheorem{remark}[theorem]{Remark}

\newcommand{\E}{\mathbb{E}}

\newcommand{\R}{\mathbb{R}}

\newcommand{\braket}[1]{\left\langle #1\right\rangle}

\title{\textbf{A note on Latała's argument in SK model}}
\author{
  Seiichiro Kusuoka\thanks{Department of Mathematics and Mathematical Sciences, Graduate School of Science, Kyoto University. E-mail: \texttt{kusuoka@math.kyoto-u.ac.jp}}
  \and
  Shuta Nakajima\thanks{Department of Mathematics, Keio University. E-mail: \texttt{njima@keio.jp}}
}

\date{}

\begin{document}

\maketitle

\begin{abstract}
In this note, we consider the Sherrington--Kirkpatrick model with deterministic external field. Let $q=q(\beta,h)$ denote the solution of the replica-symmetric self-consistency equation
\[
q=\mathbb E\tanh^2\!\left(h+\beta\sqrt q\,Z\right),
\qquad Z\sim N(0,1),
\]
where $\beta$ and $h$ are inverse temperature and external field, respectively.  By refining Latała’s argument, previously limited to \(\beta<\tfrac12\), and using the Kearns--Saul inequality, we prove overlap concentration and convergence of the free energy to the replica symmetric formula with error \(O(N^{-1})\) whenever
\[
\beta^2\frac{q}{\operatorname{arctanh}q}<1.
\]
Note that for any $\beta<1$ and $h\in \mathbb R$, the condition above is satisfied. Moreover, for every nonzero $h$, this region contains a nonempty interval with $\beta>1$.
\end{abstract}

\section{Model and main result}
In the present paper, we study the SK model, one of the classical models of spin-glass theory, defined as follows.  Let 
\(
\Sigma_N=\{-1,1\}^N,
\)  
and  $(g_{ij})_{1\leq i<j\leq N}$ be independent standard normals.  Fix an external field \(h\in\mathbb R\) and an inverse temperature $\beta\geq 0.$  For $\sigma\in\Sigma_N$, define the SK Hamiltonian by 
\begin{equation}\label{eq:SK-Hamiltonian}
H_N(\sigma)
:=
\frac{\beta}{\sqrt{N}}
\sum_{1\leq i<j\leq N}g_{ij}\sigma_i\sigma_j
+
h\sum_{i=1}^N\sigma_i.
\end{equation}
Write
\[
Z_N:=\sum_{\sigma\in\Sigma_N}e^{H_N(\sigma)},
\qquad
\phi_N:=\frac1N\E\log Z_N,
\]
and denote Gibbs expectation by $\braket{\cdot}$, i.e.
\[
\braket{F} := \frac{1}{Z_{N}}\sum _{\sigma \in \Sigma _N} F(\sigma ) e^{H_{N}(\sigma)}
\]
for a function $F$ on $\Sigma _N$.
This $Z_N$ is called the partition function, hence $\phi_N=N^{-1}\mathbb E \log Z_N$ stands for the expected free energy up to multiplication of constant.

For replicas $\sigma^a,\sigma^b$ sampled independently from the Gibbs measure conditional on the disorder, define the overlap
\[
R_{ab}:=\frac1N\sum_{i=1}^N\sigma_i^a\sigma_i^b.
\]
Let $q=q(\beta,h)\in[0,1]$ be a solution of\footnote{Existence follows because the right-hand side defines a continuous map from $[0,1]$ into itself. In the parameter region considered below, the relevant solution is unique by Guerra's result; see \cite{Talagrand2}.}
\begin{equation}\label{eq:fixed-point}
q=\E[\tanh\!\left(h+\beta\sqrt q\,Z\right)^2],
\qquad Z\sim N(0,1).
\end{equation}
For finite $\beta$ and $h$, one has $q<1$. Define
\[
\kappa(q):=
\begin{cases}
\dfrac{q}{\operatorname{arctanh}q},&q\neq0,\\[1ex]
1,&q=0,
\end{cases}
\qquad
\rho(\beta,h):=\beta^2\kappa(q(\beta,h)).
\]
We define the replica-symmetric free energy by
\[
\phi^{\mathrm{RS}}(\beta,h)
:=
\log 2
+
\E\log\cosh\!\left(h+\beta\sqrt q\,Z\right)
+
\frac{\beta^2}{4}(1-q)^2.
\]

We say that replica symmetry holds if
\[
\lim_{N\to\infty}\phi_N=\phi^{\mathrm{RS}}(\beta,h).
\]

When \(h=0\), replica symmetry is known to hold throughout the region \(\beta<1\) (see \cite{AizenmanLebowitzRuelle}), whereas it fails for \(\beta>1\) (see \cite{Guerra}).   For a general external field and sufficiently small \(\beta\), replica symmetry was later proved by Talagrand in \cite{Talagrand2} using the cavity method. Latała subsequently extended the result to the regime \(\beta<1/2\) by  a simple interpolation argument inspired by Guerra's argument \cite{Guerra}.

It is therefore natural to ask whether the reduction of the range from \(\beta<1\) to \(\beta<1/2\) is intrinsic to its interpolation method. Our purpose is to show that it is not the case.  Our proof is inspired by Latała's argument and incorporates the coupled free energy introduced below as an additional ingredient.
\begin{theorem}\label{thm:main}
Let $\beta\geq0$ and $h\in \mathbb R$, and suppose that
\begin{equation}\label{eq:improved-region}
\rho(\beta,h) = \beta^2\kappa(q(\beta,h))
<1,
\end{equation}
Then, the overlap concentration holds in the sense that there exists a constant $C=C(\beta,h)<\infty$ such that
\begin{equation}\label{eq:overlap-concentration}
\E\braket{(R_{12}-q)^2}\leq \frac{C}{N}.
\end{equation}
Consequently, we have the quantitative estimate of the replica symmetric free energy $\phi^{\mathrm{RS}}(\beta,h)$:
\begin{equation}\label{eq:RS-free-energy}
0\leq \phi^{\mathrm{RS}}(\beta,h)-\phi_N\leq \frac{C}{N}.
 \end{equation}

\end{theorem}

\begin{remark}\label{rem:beyond-beta-one}
Since $0<\kappa(q)\leq1$, condition \eqref{eq:improved-region} contains the entire region $\beta<1$. It is strictly stronger in a nonzero external field. Indeed, if $h\neq0$, then $q(\beta,h)>0$ and hence $\kappa(q(\beta,h))<1$. Thus, $\rho(\beta,h)<1$ for all $\beta>1$ sufficiently close to $1$.
\end{remark}

The present work does not establish replica symmetry in the entire
Almeida--Thouless region,
\[
\beta^2 \mathbb{E}\!\left[\cosh\!\left(h+\beta\sqrt{q}\,Z\right)^{-4}\right]\leq 1,
\]
where replica symmetry is expected to hold in the sense that
\[
\lim_{N\to\infty}\phi_N=\phi^{\mathrm{RS}}(\beta,h).
\] 
We note that a recent preprint by Lopatto \cite{Lap26} claims to resolve this question using a dynamic-programming approach developed in \cite{Jagannath}. See also \cite{BrenneckeYau2021}, which establishes a similar, though slightly stronger, condition via the conditional second-moment method introduced in \cite{Bolthausen}.

In the rest of the paper, we give the proof of Theorem~\ref{thm:main}. We also note that we provide a formalization of the proof in Lean 4 \cite{brabra}.

\begin{figure}[h]
  \centering
  \includegraphics[width=\textwidth]{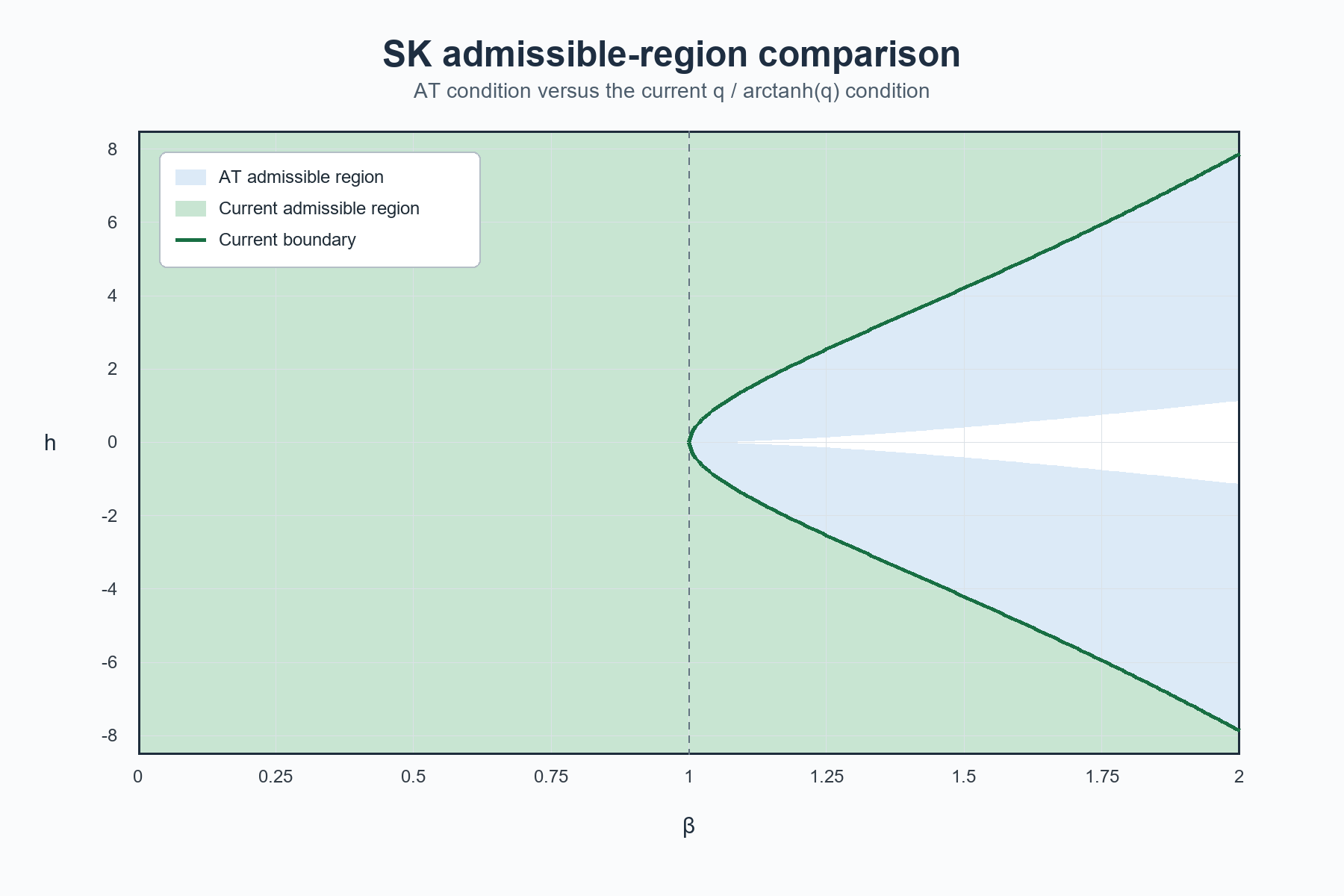}
  \caption{Comparison of the Almeida--Thouless condition 
  and the condition \eqref{eq:improved-region}.}
  \label{fig:AT-current}
\end{figure}

\section{Proof of the main theorem}

We use a smart path interpolation introduced by Guerra \cite{Guerra} (see also \cite{Talagrand2}). Let $(z_i)_{i\leq N}$ be independent standard Gaussian random variables, independent of $(g_{ij})$.  Define the smart-path Hamiltonian
\begin{equation}\label{eq:smart-path}
H_{N,t}(\sigma)
:=
\frac{\beta\sqrt t}{\sqrt N}
\sum_{i<j}g_{ij}\sigma_i\sigma_j
+
\sum_{i=1}^N
\left(h+\beta\sqrt{(1-t)q}\,z_i\right)\sigma_i,
\qquad 0\leq t\leq1.
\end{equation}
At $t=1$, this is the SK Hamiltonian \eqref{eq:SK-Hamiltonian}; at $t=0$, the sites are independent.

Let $\braket{\cdot}_t$ denote the Gibbs expectation associated with \eqref{eq:smart-path}, i.e.
\[
\braket{F}_t := \frac{1}{Z_{N,t}}\sum _{\sigma \in \Sigma _N} F(\sigma ) e^{H_{N,t}(\sigma)}
\]
where $Z_{N,t} := \sum _{\sigma \in \Sigma _N} e^{H_{N,t}(\sigma)}$, 
and set
\[
\nu_t[F]:=\E\braket{F}_t.
\]
We write
\[
\phi _N(s) := \frac{1}{N} \E \log Z_{N,s}
\]
For replicas $(\sigma^a)_a$ write
\[
Q_{ab}:=R_{ab}-q = N^{-1} \sum_{i=1}^N \sigma_i^a \sigma_i^b-q .
\]
\subsection{Endpoint estimate}
At $t=0$, conditional on $(z_i)_{i\leq N}$,
\[
m_i := \braket{\sigma_i}_0
=
\tanh\!\left(h+\beta\sqrt q\,z_i\right).
\]
For two independent replicas $\sigma^1,\sigma^2$, the variables
$X_i:=\sigma_i^1\sigma_i^2$ are independent and identically distributed under the annealed law $\mathbb E \langle \cdot\rangle_0$,
take values in $\{-1,1\}$, and satisfy
\[
\E\langle \sigma_i^1\sigma_i^2\rangle_0=\E m_i^2=\mathbb E[\tanh\!\left(h+\beta\sqrt q\,z_i\right)^2]=q,
\]
due to \eqref{eq:fixed-point}. In particular,
\[
\mathbb P(X_i=1)=\frac{1+q}{2},
\qquad
\mathbb P(X_i=-1)=\frac{1-q}{2}.
\]

We now  prove\footnote{Note that the Kearns--Saul inequality in its $\{-1,1\}$-valued form \cite{Schlemm} also gives \eqref{eq:Kearns-Saul}.}
\begin{equation}\label{eq:Kearns-Saul}
\E\braket{e^{u(X_i-q)}}_0
\leq \exp\left(\frac{\kappa(q)u^2}{2}\right),
\qquad u\in\R.
\end{equation}
 By the continuity of both sides of \eqref{eq:Kearns-Saul} in $q\in [0,1]$, it suffices to consider the case that \(0< q< 1\). Letting $a:= {\operatorname{arctanh}q}$, we obtain
\begin{align*}
\E \braket{e^{u(X_i-q)}}_0
&=e^{-qu} \left( \frac{1+q}{2} e^u + \frac{1-q}{2} e^{-u} \right) \\
&=e^{-qu} \bigl(\cosh u +q \sinh u\bigr) \\
&=e^{-qu} \frac{\cosh(u+a)}{\cosh a}.
\end{align*}
Let $x=u+a$ and define \(F(t)=\log\cosh(\sqrt t)\) for \(t\geq0\). For \(t>0\), by writing \(r=\sqrt t\) we have,
\[
F'(t)=\frac{\tanh r}{2r},\qquad
F''(t)=\frac{r\operatorname{sech}^2r-\tanh r}{4r^3}.
\]
Let
\[
h(r)=\tanh r-r\operatorname{sech}^2r.
\]
Then \(h(0)=0\) and
\[
h'(r)=2r\operatorname{sech}^2r\,\tanh r\geq0,
\]
so \(h(r)\geq0\), and hence \(F''(t)\leq0\). Thus \(F\) is concave, and
\[
F(x^2)\leq F(a^2)+F'(a^2)(x^2-a^2).
\]
Using
\[
F(x^2)=\log\cosh x,\qquad
F(a^2)=\log\cosh a,\qquad
F'(a^2)=\frac{\tanh a}{2a},
\]
and $q(x-a)+\frac{q}{2a}(x-a)^2
=\frac{q}{2a}(x^2-a^2)$, we obtain  
\[
{\;
\log\cosh x
\leq
\log\cosh a+q(x-a)+\frac{q}{2a}(x-a)^2
\; }.
\]
This together with \(x=u+a\) gives
\[
e^{-qu}\frac{\cosh(u+a)}{\cosh a}
\leq
\exp\left(\frac{q}{2a}(x-a)^2\right)
=
\exp\left(\frac{q}{2a}u^2\right)
=
\exp\left(\frac{\kappa(q)u^2}{2}\right).
\]
Therefore, we obtain \eqref{eq:Kearns-Saul}.

Since $H_{N,0} (\sigma ) = \sum _{i=1}^N \left( h+ \beta \sqrt{q} z_i \right) \sigma _i$, $\{ \sigma _i\}$ are independent with respect to the Gibbs measure at $t=0$.
This fact and \eqref{eq:Kearns-Saul} imply
\begin{equation}\label{eq:subgaussian-endpoint}
\E\braket{
\exp\left(
\frac{u}{\sqrt N}\sum_{i=1}^N(\sigma_i^1\sigma_i^2-q)
\right)
}_0
\leq \exp\left(\frac{\kappa(q)u^2}{2}\right).
\end{equation}
Integrating both sides of \eqref{eq:subgaussian-endpoint} by $\exp \left( -\frac{u^2}{2\Lambda}\right) du$ on ${\mathbb R}$ gives, whenever $\kappa(q)\Lambda<1$,
\begin{equation}\label{eq:endpoint-quadratic}
\E\braket{\exp\left(\frac{\Lambda N}{2} Q_{12}^2\right)}_0
\leq \frac{1}{\sqrt{1-\kappa(q)\Lambda}}.
\end{equation}

\subsection{Uniform quadratic-coupling estimate}\label{sec:quadratic}

We now give a key estimate.

\begin{proposition}\label{prop:quadratic}
Suppose that $\rho=\rho(\beta,h)<1$, and set
\begin{equation}\label{eq:lambda-star}
\lambda_*:=\frac{\kappa(q)^{-1}-\beta^2}{4}
=\frac{1-\rho}{4\kappa(q)}.
\end{equation}
For every $N\geq1$ and $t\in[0,1]$,
\begin{equation}\label{eq:quadratic-estimate}
\mathbb E\left[\log \left
\langle
\exp\left(\lambda_* N(R_{12}-q)^2\right)\right\rangle_t \right]
\leq \frac12
\exp\left(\frac{2\rho}{1-\rho}\right)
\log\left(\frac{2}{1-\rho}\right).
\end{equation}
\end{proposition}

\begin{proof}
Fix $t\in[0,1]$. For $\Lambda\geq0$, introduce the  coupled free energy
\[
\phi_N(s,\Lambda)
:=
\frac{1}{2N}
\E\left[\log \left(
\sum_{\sigma^1,\sigma^2\in\Sigma_N}
\exp\left\{
{H}_{N,s}(\sigma^1)
+
{H}_{N,s}(\sigma^2)
+
\frac{\Lambda N}{2}Q_{12}^2
\right\}\right)\right].
\]
Clearly,
\(
\phi_N(s,0)=\phi_N(s).
\)  Let $C_s(\sigma,\tau)$ be the covariance of ${H}_{N,s}(\sigma)$ and ${H}_{N,s}(\tau)$. Then, we have
\[
C_s(\sigma,\tau)
=\beta^2\left(
\frac{s}{N}\sum_{i<j}\sigma_i\tau_i\sigma_j\tau_j
+
(1-s)q\sum_{i=1}^N\sigma_i\tau_i\right).
\]
For $\sigma,\tau\in\Sigma_N$, write
\[
R(\sigma,\tau):=\frac1N\sum_{i=1}^N\sigma_i\tau_i.
\]
Using
\begin{equation}\label{eq:R^2}
\sum_{i<j}\sigma_i\tau_i\sigma_j\tau_j
=
\frac12\left(N^2R(\sigma,\tau)^2-N\right),
\end{equation}
we obtain
\begin{equation}\label{eq:covariance-derivative}
\dot{C}_s(\sigma,\tau):=\frac{\rm d}{{\rm d}s} C_s(\sigma,\tau)
=
\frac{\beta^2 N}2\left((R(\sigma,\tau)-q)^2-q^2\right)-\frac{\beta^2}{2}.
\end{equation}

Since
\[
\frac{\rm d}{{\rm d}s} \phi_N(s) = \frac{\beta }{2N} \E\left[\frac{1}{Z_{N,s}} \sum _{\sigma \in \Sigma _N} e^{H_{N,s}(\sigma )} \left( \frac{1}{\sqrt{sN}}
\sum_{i<j}g_{ij}\sigma_i\sigma_j
- \sqrt{\frac{q}{1-s}}
\sum_{i=1}^N
z_i \sigma_i \right) \right],
\]
by Gaussian integration by parts (see \cite[Vol.~1,~Sec.~A.4]{Talagrand2}), we have 
\begin{align*}
\frac{\rm d}{{\rm d}s} \phi_N(s) &= \frac{\beta }{2N} \sum _{\sigma \in \Sigma _N} \sum_{i<j} \sigma _i \sigma _j \E\left[ \frac{e^{H_{N,s}(\sigma )}}{Z_{N,s}} \frac{\beta}{N} \sigma _i \sigma _j -\frac{e^{H_{N,s}(\sigma )}}{Z_{N,s}^2}  \frac{\beta}{N} \sum _{\tau \in \Sigma _N} \tau _i \tau _j e^{H_{N,s}(\tau )}\right] \\
&\quad - \frac{\beta }{2N} \sum _{\sigma \in \Sigma _N} \sum_{i=1}^N \sigma _i \E\left[ \frac{e^{H_{N,s}(\sigma )}}{Z_{N,s}} \beta q \sigma _i -\frac{e^{H_{N,s}(\sigma )}}{Z_{N,s}^2} \beta q \sum _{\tau \in \Sigma _N} \tau _i e^{H_{N,s}(\tau )}\right] \\
&= \frac{\beta ^2 (N-1)}{4N} -\frac{\beta ^2}{2} \sum _{\sigma \in \Sigma _N} \E\left[ \sum _{\sigma , \tau \in \Sigma _N} \frac{e^{H_{N,s}(\sigma )}}{Z_{N,s}} \frac{e^{H_{N,s}(\tau )}}{Z_{N,s}} \cdot \frac{1}{N^2} \sum_{i<j} \sigma _i \sigma _j \tau _i \tau _j \right] \\
&\quad - \frac{\beta ^2 q}{2} + \frac{\beta ^2 q}{2} \E\left[ \sum _{\sigma , \tau \in \Sigma _N} \frac{e^{H_{N,s}(\sigma )}}{Z_{N,s}} \frac{e^{H_{N,s}(\tau )}}{Z_{N,s}} \cdot \frac{1}{N}\sum _{i=1}^N \sigma _i \tau _i \right] .
\end{align*}
Hence, in view of \eqref{eq:R^2} it holds that
\begin{equation}\label{eq:p-derivative}
\frac{\rm d}{{\rm d}s} \phi_N(s)
=\frac{\beta^2}{4}\left(
(1-q)^2
-
\nu_s[Q_{12}^2]\right).
\end{equation}

Let $\widehat\nu_{s,\Lambda}$ denote expectation with respect to two independent pairs
$(\sigma^1,\sigma^2)$ and $(\sigma^3,\sigma^4)$ sampled from the coupled Gibbs measure defining $\phi_N(s,\Lambda)$, i.e. 
\[
\widehat\nu_{s,\Lambda}[F] := \E\left[\frac{ \sum_{\sigma^1,\sigma^2,\sigma^3,\sigma^4\in\Sigma_N} F(\sigma ^1, \sigma ^2, \sigma ^3,\sigma ^4) 
e^{H_{N,s}^\Lambda (\sigma ^1, \sigma ^2)} 
e^{H_{N,s}^\Lambda (\sigma ^3, \sigma ^4)}
}{\left(
\sum_{\sigma^1,\sigma^2\in\Sigma_N}
e^{H_{N,s}^\Lambda (\sigma ^1, \sigma ^2)} \right) ^2 }\right]
\]
where $H_{N,s}^\Lambda (\sigma ^1, \sigma ^2) := {H}_{N,s}(\sigma^1) + {H}_{N,s}(\sigma^2) + \frac{\Lambda N}{2}Q_{12}^2$.
Denote $\sum_{\sigma^1,\sigma^2\in\Sigma_N} e^{H_{N,s}^\Lambda (\sigma ^1, \sigma ^2)}$ by $Z_{N,s}^\Lambda$.
Similarly to \eqref{eq:p-derivative} we are able to calculate $\frac{\partial}{\partial s} \phi_N(s,\Lambda)$ as
\begin{align*}
&\frac{\partial}{\partial s} \phi_N(s,\Lambda) \\
&= \frac{\beta }{4N} \sum _{\sigma ^1 ,\sigma ^2 \in \Sigma _N} \sum_{i<j} (\sigma _i^1 \sigma _j^1 + \sigma _i^2 \sigma _j^2 ) \E\left[ \frac{e^{H_{N,s}^\Lambda (\sigma ^1, \sigma ^2)} }{Z_{N,s}^\Lambda } \frac{\beta}{N} (\sigma _i^1 \sigma _j^1 + \sigma _i^2 \sigma _j^2 ) \right. \\
&\quad \hspace{6cm} \left. -\frac{e^{H_{N,s}^\Lambda (\sigma ^1, \sigma ^2)}}{(Z_{N,s}^\Lambda )^2}  \frac{\beta}{N} \sum _{\tau ^1, \tau ^2 \in \Sigma _N} (\tau _i^1 \tau _j^1 +\tau _i^2 \tau _j^2 ) e^{H_{N,s}^\Lambda (\tau ^1, \tau ^2)}\right] \\
&\quad - \frac{\beta }{4N} \sum _{\sigma ^1 ,\sigma ^2 \in \Sigma _N} \sum_{i=1}^N (\sigma _i^1 + \sigma _i^2 ) \E\left[ \frac{e^{H_{N,s}^\Lambda (\sigma ^1, \sigma ^2)} }{Z_{N,s}^\Lambda } \beta q (\sigma _i^1 + \sigma _i^2 ) \right. \\
&\quad \hspace{6cm} \left.  -\frac{e^{H_{N,s}^\Lambda (\sigma ^1, \sigma ^2)}}{(Z_{N,s}^\Lambda )^2} \beta q \sum _{\tau ^1, \tau ^2 \in \Sigma _N} (\tau _i^1 + \tau _i^2) e^{H_{N,s}^\Lambda (\tau ^1, \tau ^2)} \right] \\
&=
\frac{1}{4N}
\widehat\nu_{s,\Lambda}
\left[
2\dot C_s(\sigma^1,\sigma^1)
+
2\dot C_s(\sigma^1,\sigma^2)
-
\sum_{{a\in\{1,2\},b\in\{3,4\}}}
\dot C_s(\sigma^a,\sigma^b)
\right].  
\end{align*}
By symmetry of the coupled Gibbs measure under exchange within each replica pair, the four cross terms have the same expectation.  Substitution of \eqref{eq:covariance-derivative} gives
\begin{equation}\label{eq:coupled-s-derivative}
\frac{\partial}{\partial s} \phi_N(s,\Lambda)
=
\frac{\beta^2}{4}\left((1-q)^2
+
\widehat\nu_{s,\Lambda}[Q_{12}^2]
-
2\widehat\nu_{s,\Lambda}[Q_{13}^2]\right).
\end{equation}
Moreover,
\begin{equation}\label{eq:Lambda-derivative}
\partial_{\Lambda} \phi_N(s,\Lambda)
=
\frac14\widehat\nu_{s,\Lambda}[Q_{12}^2].
\end{equation}

Let $\lambda_*$ be as in \eqref{eq:lambda-star}. For $0\leq s\leq t$, set
\[
\Lambda_s:=2\lambda_*+\beta^2(t-s),
\qquad
F_N(s,\Lambda ):=\phi_N(s,\Lambda )-\phi_N(s),
\qquad
D_N(s) := F_N(s,\Lambda _s).
\]
Since $\frac{\rm d}{{\rm d}s} \Lambda_s = -\beta^2$, \eqref{eq:p-derivative}, \eqref{eq:coupled-s-derivative}, and \eqref{eq:Lambda-derivative} imply
\begin{equation}\label{eq:D'}
\frac{\rm d}{{\rm d}s} D_N(s)
=\beta^2\left(
-\frac12\widehat\nu_{s,\Lambda_s}[Q_{13}^2]
+
\frac14\nu_s[Q_{12}^2]\right)
\leq
\frac{\beta^2}{4}\nu_s[Q_{12}^2].
\end{equation}
On the other hand, 
\[
F_N(s,\Lambda ) = 
\frac{1}{2N}\mathbb E\log
\left\langle\exp\left(\frac{\Lambda N}{2}Q_{12}^2\right)\right\rangle_s.
\]
Hence, for each $s\in [0,t]$ the function $\Lambda \mapsto F_N(s,\Lambda )$ on $[0,\infty )$ is nonnegative and convex because its second derivative is a nonnegative tilted variance. Moreover, $F_N(s,0)=0$ and
\[
\left( \frac{\partial}{\partial \Lambda} F_N \right) (s,0+)=\frac14\nu_s[Q_{12}^2].
\]
Therefore, by convexity,
\[
\frac14\nu_s[Q_{12}^2]
\leq
\frac{F_N(s,\Lambda_s)}{\Lambda_s} = \frac{D_N(s)}{\Lambda_s}.
\]
Since $\Lambda_s\geq2\lambda_*$, from these inequalities and \eqref{eq:D'} we conclude that
\[
\frac{\rm d}{{\rm d}s} D_N(s) \leq \frac{\beta^2 D_N(s)}{2\lambda_*}.
\]
Hence, Gronwall's inequality yields
\begin{equation}\label{eq:Gronwall}
D_N(t)
\leq
\exp\left(\frac{\beta^2 t}{2\lambda_*}\right)D_N(0).
\end{equation}

It remains to estimate $D_N(0)$. At $s=0$, \eqref{eq:endpoint-quadratic} implies that, whenever $\kappa(q)\Lambda<1$,
\[
\E\braket{\exp\left(\frac{\Lambda N}{2}Q_{12}^2\right)}_0
\leq \frac{1}{\sqrt{1-\kappa(q)\Lambda}}.
\]
By applying \eqref{eq:lambda-star} and recalling $\rho = \beta ^2 \kappa (q)$, for $0\leq t\leq1$, we have
\[
1-\kappa(q)(2\lambda_*+\beta^2t)
\geq \frac{1-\rho}{2}>0.
\]
Since this inequality and Jensen's inequality give
\[
\E\log\braket{
\exp\left(
\frac{(2\lambda_*+\beta^2t)N}{2}Q_{12}^2
\right)
}_0 \leq
\frac12
\log\left(\frac{1}{1-\kappa(q)(2\lambda_*+\beta^2t)}\right)
\leq
\frac12\log\left(\frac{2}{1-\rho}\right), 
\]
it holds that
\begin{equation}\label{eq:D0}
2N D_N(0) \leq \frac12\log\left(\frac{2}{1-\rho}\right)
\end{equation}

At $s=t$, we have $\Lambda_{t}=2\lambda_*$. Hence
\(
2N D_N(t)
=
\E\log\braket{
e^{\lambda_*NQ_{12}^2}
}_t.
\) 
Combining this with \eqref{eq:Gronwall} and \eqref{eq:D0} and using
\[
\frac{\beta^2}{2\lambda_*}=\frac{2\rho}{1-\rho}
\]
yields
\[
\E\log\braket{
e^{\lambda_*NQ_{12}^2}
}_t
\leq
\frac12
\exp\left(\frac{\beta^2t}{2\lambda_*}\right)
\log\left(\frac{2}{1-\rho}\right)
\leq \frac12
\exp\left(\frac{2\rho}{1-\rho}\right)
\log\left(\frac{2}{1-\rho}\right),
\]
which proves \eqref{eq:quadratic-estimate}. 
\end{proof}

\subsection{Overlap concentration: Proof of \eqref{eq:overlap-concentration}}
We give the proof of \eqref{eq:overlap-concentration}.
Let $\lambda_*$ be as in \eqref{eq:lambda-star}. For $s\geq0$, define
\[
F_{N,t}(s)
:=
\frac1N\E\log\braket{e^{sNQ_{12}^2}}_t.
\]
The map $s\mapsto F_{N,t}(s)$ is convex, $F_{N,t}(0)=0$, and
\(
F_{N,t}'(0+)=\nu_t[Q_{12}^2].
\) 
Set
\[
C_0:=\frac12
\exp\left(\frac{2\rho}{1-\rho}\right)
\log\left(\frac{2}{1-\rho}\right).
\]
By the convexity of $F_{N,t}$ and Proposition~\ref{prop:quadratic}, we have
\begin{equation}\label{eq: overlap concentratino with t}
    \nu_t[Q_{12}^2]
\leq
\frac{F_{N,t}(\lambda_*)-F_{N,t}(0)}{\lambda_*}
\leq
\frac{C_0}{\lambda_* N}.
\end{equation}
Letting $t=1$ in this inequality yields
\[
\E\braket{(R_{12}-q)^2} = \nu_1[Q_{12}^2]
\leq
\frac{C_0}{\lambda_*N}.
\]

\subsection{Replica symmetric free energy: Proof of \eqref{eq:RS-free-energy}}
The proof of \eqref{eq:RS-free-energy} is similar to Latała's argument. 
Recall that
\(
\phi_N(t)=\frac1N\E\log Z_{N,t},
\) 
where $Z_{N,t}$ is the partition function associated with \eqref{eq:smart-path} and that
\begin{equation}\label{eq:sum-rule}
\frac{\rm d}{{\rm d}t} \phi_N(t)
=
\frac{\beta^2}{4}(1-q)^2
-
\frac{\beta^2}{4}\nu_t\left[(R_{12}-q)^2\right]
\end{equation}
(see \eqref{eq:p-derivative}).
Since $H_{N,0} (\sigma ) = \sum _{i=1}^N \left( h+ \beta \sqrt{q} z_i \right) \sigma _i$, $\{ \sigma _i\}$ are independent with respect to the Gibbs measure at $t=0$.
Hence, it holds that
\begin{equation}\label{eq:initial-pressure}
\phi_N(0)
=
\log 2
+
\E\log\cosh\!\left(h+\beta\sqrt q\,Z\right)
\end{equation}
where $Z$ is a random variable with the standard normal law.
Since $\phi_N(1)=\phi_N$, integrating both sides of \eqref{eq:sum-rule} gives
\begin{align*}
    \phi_N
&=
\log 2
+
\E\log\cosh\!\left(h+\beta\sqrt q\,Z\right)
+
\frac{\beta^2}{4}(1-q)^2
-
\frac{\beta^2}{4}\int_0^1\nu_t[Q_{12}^2]\,dt\\
&=
\phi^{\mathrm{RS}}(\beta,h)
-
\frac{\beta^2}{4}\int_0^1\nu_t[Q_{12}^2]\,dt.
\end{align*}
This equality and \eqref{eq: overlap concentratino with t} yield
\[
0\leq \phi^{\mathrm{RS}}(\beta,h)-\phi_N
\leq \frac{\beta^2C_0}{4\lambda_*N}.
\]
\section*{Acknowledgements}
S.K.\ acknowledges support from JSPS KAKENHI Grant Numbers
JP22H00099 and JP23K20801.
S.N.\ acknowledges support from JSPS KAKENHI Grant Numbers
JP24K16937 and JP25K00911.

\section*{Statement on the Use of Artificial Intelligence}

For the generalization of Lata{\l}a's argument, the coupled free energy was suggested by the authors. ChatGPT 5.5 Plus was used to assist with extending the calculations, and Codex was used to perform sanity checks. The Lean~4 formalization was developed with substantial assistance from Codex and Aristotle. The authors take full responsibility for the content of this work.

\end{document}